\documentclass{article}
\usepackage[dvipdfmx]{graphicx}
\usepackage{graphics}
\usepackage{xcolor}

\usepackage{array}
\usepackage{booktabs}
\usepackage{amsthm}
\usepackage{amssymb}
\usepackage{amsfonts}
\usepackage{mathrsfs}
\usepackage{amsmath}
\usepackage[all]{xy}
\usepackage{amscd}
\usepackage{enumerate}
\usepackage{flafter}

\allowdisplaybreaks

\theoremstyle{definition}
\newtheorem{definition}{Definition}[section]
\newtheorem{remark}[definition]{Remark}
\newtheorem{example}[definition]{Example}
\theoremstyle{plain}
\newtheorem{proposition}[definition]{Proposition}
\newtheorem{lemma}[definition]{Lemma}
\newtheorem{theorem}[definition]{Theorem}
\newtheorem{corollary}[definition]{Corollary}

\def\dim{\mathop{\mathrm{dim}}\nolimits}

\def\Ker{\mathop{\mathrm{Ker}}\nolimits}
\def\Coker{\mathop{\mathrm{Coker}}\nolimits}

\newcommand{\Q}{{\mathbb{Q}}}

\newcommand{\Mat}{{\rm Mat}}
\newcommand{\PP}{\mathcal{P}}
\newcommand{\NF}{{\rm NF}}

\DeclareMathOperator{\gcdop}{gcd}

\newcommand{\lcm}{{\rm lcm}}
\newcommand{\lex}{lex}
\newcommand{\ideal}[1]{\left\langle #1 \right\rangle}

\DeclareMathOperator{\dimQ}{dim_{\Q}}
\DeclareMathOperator{\inop}{in}

\newcommand{\LM}{\operatorname{LM}}
\newcommand{\AG}{\mathrm{AG}}

\newcommand{\Z}{\mathbb{Z}}
\newcommand{\K}{\mathbb{K}}

\newcommand{\Fitt}{\operatorname{Fitt}}

\newcommand{\safeincludegraphics}[2][]{%
  \IfFileExists{#2}{\includegraphics[#1]{#2}}{\fbox{\texttt{Missing figure: #2}}}%
}
\newcommand{\pcc}[2]{\mbox{$\begin{array}{c}
\safeincludegraphics[scale=#2]{#1.pdf}
\end{array}$}}
\newcommand{\pc}[2]{\mbox{$\begin{array}{c}
\safeincludegraphics[scale=#2]{#1.eps}
\end{array}$}}

\title{Gr\"obner Bases for Alexander Fitting Ideals of Links}
\author{Takefumi Nosaka\footnote{E-mail address: {\tt nosaka@math.titech.ac.jp}}
 \and Kotaro Shimizu\footnote{E-mail address: {\tt shimizu.k.38ee@m.isct.ac.jp}}}
\date{}

\begin{document}

\maketitle

\begin{abstract}
Multivariable Alexander theory produces Fitting ideals without preferred
generators.  After fixing a coefficient field and a monomial order, we
represent these ideals by reduced Gr\"obner bases of their polynomial
contractions, obtaining Alexander--Gr\"obner invariants of links.
Over \(\mathbb Q\), for the maximal free-abelian coefficient system of
a connected compact oriented \(3\)-manifold whose nonempty boundary is
a disjoint union of tori, we deduce determinant-divisor reciprocity 
from Blanchfield duality.
We compute these invariants for torus links and low-crossing links, and
we determine \(\AG_2\) for a two-component pretzel family.
\end{abstract}

\begin{center}
\normalsize
\baselineskip=11pt
{\bf Keywords.}\\
Gr\"obner bases, multivariable Alexander polynomial, Fitting ideals, Fox calculus, link invariants
\end{center}

\begin{center}
\normalsize
\baselineskip=11pt
{\bf 2020 Mathematics Subject Classification.}\\
Primary 57K10; Secondary 57K14, 13P10, 57M25, 57M05
\end{center}

\section{Introduction}\label{sec:introduction}

Multivariable Alexander theory associates to a finitely presented group
\(G\), equipped with an epimorphism
\(\phi\colon G\twoheadrightarrow\Z^\ell\), a module over the Laurent
polynomial ring
$
 \Lambda=\K[t_1^{\pm1},\dots,t_\ell^{\pm1}],
$
and hence a sequence of Fitting ideals, traditionally called elementary
ideals.  For link complements, these are classical invariants
\cite{CF,FT54}.  With the relative Alexander module used in this paper,
the usual one-variable Alexander polynomial is recovered from the torsion
part in the knot case, while for a classical link with \(\ell>1\), the
first elementary ideal \(E_1\) and the multivariable Alexander polynomial
determine one another up to the usual units of \(\Lambda\)
(Examples~\ref{ex:l1} and~\ref{ex:k1}).  Higher elementary ideals may be
proper and nonprincipal.  Their greatest common divisors retain the
divisorial, or height-one, information, whereas the full ideals may also
carry information in higher codimension.  A basic practical difficulty is
that an ideal has no distinguished finite generating set.

The main purpose of this paper is to replace arbitrary generating sets of
Alexander Fitting ideals by unique finite normal forms, relative to
explicit algebraic choices.  Given a finite presentation of \(G\), let \(A\) be the
Fox matrix associated with \(\phi\), define
\[
 \mathcal A(G,\phi):=\Coker(A^{\mathsf T}),
 \qquad
 E_k(G,\phi):=\Fitt_k(\mathcal A(G,\phi))\subset\Lambda,
\]
and call \(\mathcal A(G,\phi)\) the relative Alexander module.  Its
topological interpretation shows that \(\mathcal A(G,\phi)\), and hence
every \(E_k(G,\phi)\), depends only on the pair \((G,\phi)\)
(Theorem~\ref{thm:CF-invariance}).  Fix the polynomial subring
$
 P=\K[t_1,\dots,t_\ell]\subset\Lambda
$
and a monomial order on \(P\).  Since \(\Lambda\) is obtained from \(P\)
by inverting the variables, every ideal \(I\subset\Lambda\) is recovered
from its contraction \(I\cap P\) (Lemma~\ref{lem:contraction-injective}).
We therefore define
\[
 \AG_k(G,\phi)
 :=\mathcal G_P\bigl(E_k(G,\phi)\cap P\bigr),
\]
where \(\mathcal G_P(J)\) denotes the reduced Gr\"obner basis of
\(J\subset P\) with respect to the fixed order
(Definition~\ref{def:AG}).  Thus \(\AG_k(G,\phi)\) is not a coarser replacement for
\(E_k(G,\phi)\): it is a unique finite representative of the original
Laurent ideal once the coefficient field, the ordered basis of
\(\Z^\ell\), and the monomial order have been fixed.  For oriented ordered
spherical codimension-two links, the positive meridians determine the
coefficient system, and the construction gives link invariants
(Subsection~\ref{sec:links} and
Corollary~\ref{cor:link-invariant}).  
For classical links with more than one component, \(\AG_1\) carries
precisely the image, in the fixed coefficient field, of the usual
multivariable Alexander-polynomial information, so the main new phenomena occur for
\(k>1\) (Example~\ref{ex:k1}).

We first examine the higher invariants through exact computations for
low-crossing classical links over \(\Q\).  The resulting bases include
proper, nonzero, nonprincipal second elementary ideals and display
substantial variation in both size and algebraic form
(Example~\ref{ex:k2} and
Table~\ref{tab:AG2_8cross_3comp_filtered}).  More importantly, among the
chosen oriented ordered two- and three-component representatives with at
most \(10\) crossings, three pairs are exhibited explicitly and fourteen
additional pairs are listed for which the multivariable Alexander
polynomials agree up to Laurent units but the values of \(\AG_2\) differ
(Example~\ref{ex:AG-refines-Alexander} and
Remark~\ref{rem:further-Alexander-collisions}).  
These computations show,
within the stated data set and conventions, that \(\AG_2\) is not
determined by the multivariable Alexander polynomial
(Example~\ref{ex:AG-refines-Alexander}).  We also record a separate
computational experiment for certain three-component pretzel
links; it is included as exploratory evidence and is not used in the
proofs of the subsequent results
(Example~\ref{ex:3pretzel-experiment}).

The main uniform calculations concern two infinite families.  For the
torus links
$  T(\ell n,\ell m),
 \qquad \ell>1,\qquad \gcdop(m,n)=1$,
we determine \(E_k\), its contraction to the polynomial ring, and the
reduced Gr\"obner basis \(\AG_k\) for every \(1\le k\le\ell\)
(Proposition~\ref{prop:toruslink-Ek} and
Corollary~\ref{cor:toruslink-groebner}).  For the two-component pretzel
family $
 P(2k,2r-1,2m)$,
we identify the contraction of \(E_2\) with an explicit three-generated
ideal in \(\Q[x,y]\), and a cyclotomic local decomposition yields a
uniform Gr\"obner-basis construction from which \(\AG_2\) is obtained by
interreduction (Proposition~\ref{prop:pretzel-reduction},
Lemma~\ref{lem:pretzel-local-uniform}, and
Theorem~\ref{thm:pretzel-uniform}).  The cases
\(\gcdop(m,k)=1\) and \(\gcdop(m,k)>1\) are then described separately
(Theorems~\ref{thm:pretzel-g1} and~\ref{thm:pretzel-gt1}).

The reciprocal patterns visible in these computations motivate the final
section, whose role is to isolate the part of the phenomenon forced by
duality.  We work over \(\Q\) with the maximal free-abelian coefficient
system of a connected compact oriented \(3\)-manifold whose nonempty
boundary is a union of tori.  No assertion is made here for an arbitrary
epimorphism onto a free abelian group
(Remark~\ref{rem:scope-of-reciprocity}).  Blanchfield's theory implies
reciprocity for every determinant divisor of the relative Alexander module
(Theorem~\ref{thm:determinant-reciprocity}).  For a classical link, this in
particular gives reciprocity of a greatest common divisor of the elements
of \(\AG_k\) (Corollary~\ref{cor:AG-common-divisor-reciprocity}).  At the
level of full ideals, a Fitting ideal and its involutive image agree after
localization at every height-one prime, so any remaining discrepancy is
supported in codimension at least two
(Proposition~\ref{prop:codimension-two-defect}).  We do not claim full
ideal reciprocity in general.  Instead, we give an exact finite
Gr\"obner-basis criterion for it
(Proposition~\ref{prop:groebner-full-reciprocity}).  Thus the final section
explains the codimension-one part of the observed reciprocity and provides
a precise test for the stronger property appearing in the computed
examples.

The paper is organized as follows.  Section~\ref{sec:preliminaries}
reviews Fox derivatives, relative Alexander modules, Fitting ideals, and
their topological interpretation.  Section~\ref{sec:AG-definition}
introduces the Alexander--Gr\"obner invariants, specializes them to links,
and records the low-crossing and exploratory computations.
Section~\ref{sec:families} contains the explicit calculations for torus
links and two-component pretzel links.  Finally,
Section~\ref{sec:reciprocity} treats determinant-divisor reciprocity, the
codimension-two defect, and the Gr\"obner-basis criterion for full ideal
reciprocity.

\medskip
\noindent\textbf{Notation and conventions.}
The integers \(n,m,\ell\) are positive unless otherwise stated, and \(\K\)
denotes a field.  
By $\Lambda$ we mean the Laurent polynomial ring $\K[t_1^{\pm1},\dots,t_\ell^{\pm1}]$,
and by $P$ we mean the polynomial subring $\K[t_1, \dots,t_\ell]$.
Whenever Gr\"obner bases are considered, we
use the monomial order, including the order of the variables, fixed in
Remark~\ref{rem:term-order}.

\section{Fox derivatives, Fitting ideals, and topology}\label{sec:preliminaries}

This section recalls the Fox-calculus description of the relative Alexander
module, its Fitting ideals, and their topological interpretation.

\subsection{Fox derivatives, the relative Alexander module, and Fitting ideals}\label{subsec:fox}

Fix a finitely presented group \(G\) and an epimorphism
\(\phi\colon G\twoheadrightarrow \Z^\ell\), and put
\(\Lambda:=\K[t_1^{\pm1},\dots,t_\ell^{\pm1}]\).
Here \(t_i\) corresponds to the \(i\)th standard basis vector of
\(\Z^\ell\). All \(\Lambda\)-modules are left modules, and matrices are
written with respect to column vectors.

Choose a finite presentation
\[
G=\langle x_1,\dots,x_n\mid r_1,\dots,r_m\rangle,
\]
and let \(F_n\) be the free group on \(x_1,\dots,x_n\). For each \(j\),
the Fox derivative \(\partial/\partial x_j\colon\Z[F_n]\to\Z[F_n]\) is
characterized by
\[
\frac{\partial x_i}{\partial x_j}=\delta_{ij},
\qquad
\frac{\partial(uv)}{\partial x_j}
 =\frac{\partial u}{\partial x_j}
  +u\frac{\partial v}{\partial x_j}
\quad (u,v\in F_n).
\]
In particular,
\(\partial x_i^{-1}/\partial x_j=-x_i^{-1}\delta_{ij}\); see
\cite{Fox1953}.

The composite of \(F_n\twoheadrightarrow G\) with \(\phi\) induces a ring
homomorphism \(\Phi\colon\Z[F_n]\to\Lambda\). Thus, if
\(\phi(x_j)=(\alpha_1,\dots,\alpha_\ell)\), then
\(\Phi(x_j)=t_1^{\alpha_1}\cdots t_\ell^{\alpha_\ell}\). Define the Fox
matrix by
\[
A=
\left(
 \Phi\!\left(\frac{\partial r_i}{\partial x_j}\right)
\right)_{1\le i\le m,\;1\le j\le n}
\in\Mat_{m\times n}(\Lambda).
\]
With our column-vector convention, \(A^{\mathsf T}\) is a presentation
matrix. The \emph{relative Alexander module} of \((G,\phi)\) is defined by
\[
\Lambda^m\xrightarrow{\;A^{\mathsf T}\;}\Lambda^n
 \longrightarrow \mathcal A(G,\phi)
 \longrightarrow 0.
\]
Equivalently, \(\mathcal A(G,\phi)=\Coker(A^{\mathsf T})\).

Next, we review Fitting ideals.
Let \(R\) be a commutative ring and let \(M\) be an \(R\)-module with a
finite presentation
\[
R^q\xrightarrow{\;B\;}R^p
 \longrightarrow M
 \longrightarrow 0.
\]
For an integer \(r\), let \(I_r(B)\subset R\) be the ideal generated by the
\(r\times r\) minors of \(B\), with the conventions
\(I_r(B)=0\) for \(r>\min\{p,q\}\) and \(I_r(B)=R\) for \(r\le0\).
The \(k\)th Fitting ideal of \(M\) is
\[
\Fitt_k(M):=I_{p-k}(B).
\]
It is independent of the chosen finite presentation; see
\cite[\S20]{Eis1}.

For the relative Alexander module, set
\[
E_k(G,\phi)
 :=\Fitt_k(\mathcal A(G,\phi))
 =I_{n-k}(A^{\mathsf T})
 =I_{n-k}(A)
 \subset\Lambda.
\]
The last equality follows because determinantal ideals are unchanged by
transposition. We call \(E_k(G,\phi)\) the \emph{\(k\)th elementary ideal} of
\((G,\phi)\).

\subsection{Topological interpretation}\label{subsec:topological-meaning}

Let \(W_{\mathcal{P}}\) be the presentation \(2\)-complex associated with the chosen
presentation $\mathcal{P}$ of \(G\). Thus \(W_{\mathcal{P}}\) has one \(0\)-cell, \(n\) \(1\)-cells,
and \(m\) \(2\)-cells, and \(\pi_1(W_{\mathcal{P}})\cong G\). The following standard
identification also proves presentation independence; compare
\cite{CF,Fox1953,Lyn}.

\begin{theorem}\label{thm:CF-invariance}
With the local coefficient system induced by \(\phi\), there is an
isomorphism
\[
\mathcal A(G,\phi)\cong H_1(W_{\mathcal{P}},*;\Lambda).
\]
Consequently, \(\mathcal A(G,\phi)\), and hence every \(E_k(G,\phi)\),
depends only on the pair \((G,\phi)\) and not on the chosen finite
presentation.
\end{theorem}

\begin{proof}
Let \(p\colon\widetilde W_{\mathcal{P}}\to W_{\mathcal{P}}\) be the regular covering corresponding
to \(\Ker(\phi)\). The standard covering-space model identifies
\(C_*(W_{\mathcal{P}},*;\Lambda)\) with
\(C_*(\widetilde W_{\mathcal{P}},p^{-1}(*);\K)\). After choosing lifts of the
\(1\)- and \(2\)-cells, the relative cellular chain groups in degrees
\(1\) and \(2\) are identified with \(\Lambda^n\) and \(\Lambda^m\),
respectively. Fox's cellular boundary formula identifies the boundary map
from degree \(2\) to degree \(1\) with the homomorphism represented by
\(A^{\mathsf T}\). Since the relative chain group in degree \(0\) is zero,
we obtain
\(H_1(W_{\mathcal{P}},*;\Lambda)\cong\Coker(A^{\mathsf T})=\mathcal A(G,\phi)\).

By successively attaching, if necessary, cells of dimensions at least
\(3\), one obtains an Eilenberg--Mac Lane space \(K(G,1)\). Since these
attachments do not change the fundamental group, the coefficient system
extends to this space. They also leave the chain groups in degrees \(1\)
and \(2\), and the boundary map from degree \(2\) to degree \(1\),
unchanged. Hence they do not change relative homology in degree \(1\),
which is therefore determined, up to \(\Lambda\)-module isomorphism, by
\((G,\phi)\).
\end{proof}

\begin{remark}\label{rem:absolute-relative}
Since \(\phi\) is surjective, \(H_0(W_{\mathcal{P}};\Lambda)\cong\K\) via the
augmentation homomorphism. The long exact sequence of the pair gives
\[
0\longrightarrow H_1(W_{\mathcal{P}};\Lambda)
 \longrightarrow \mathcal A(G,\phi)
 \longrightarrow \ker\varepsilon
 \longrightarrow 0,
\]
where \(\varepsilon\colon\Lambda\to\K\) is defined by
\(\varepsilon(t_i)=1\). If \(\ell=1\), then
\(\ker\varepsilon=(t-1)\Lambda\cong\Lambda\); hence the sequence splits
noncanonically, and
\(\mathcal A(G,\phi)\cong H_1(W_{\mathcal{P}};\Lambda)\oplus\Lambda\).
\end{remark}

\begin{example}\label{ex:l1}
Suppose that \(\ell=1\) and \(\K=\Q\). Since
\(\Lambda=\Q[t^{\pm1}]\) is a PID, \(\mathcal A(G,\phi)\) is the direct
sum of a free module and a finitely generated torsion module. The order of
the torsion module, namely the product of its invariant factors, is the
one-variable Alexander polynomial associated with \((G,\phi)\), up to a
unit of \(\Lambda\). For the canonical epimorphism of a knot group, the
absolute Alexander module \(H_1(W_{\mathcal{P}};\Lambda)\) is torsion. Hence
Remark~\ref{rem:absolute-relative} shows that the free rank of
\(\mathcal A(G,\phi)\) is one, and the torsion order above is the usual
Alexander polynomial of the knot.
\end{example}

\section{Alexander--Gr\"obner invariants}\label{sec:AG-definition}
This section introduces Fitting-ideal Gr\"obner bases and specializes
them to oriented links.  The final subsection records low-crossing
calculations and exploratory computations for certain three-component
pretzel links.

Put \(P:=\K[t_1,\dots,t_\ell]\subset\Lambda\), and fix a monomial order on
\(P\). For an ideal \(J\subset P\), let \(\mathcal G_P(J)\) denote its
reduced Gr\"obner basis with respect to this order. 
We use the standard existence and uniqueness theorem for reduced
Gr\"obner bases, the normal-form membership criterion, and the fact
that standard monomials form a basis of the quotient; see \cite[\S15]{Eis1}.

For an ideal \(I\subset\Lambda\), write \(\PP(I):=P\cap I\) for its
contraction. The following lemma shows that contraction loses no
information.

\begin{lemma}\label{lem:contraction-injective}
Contraction \(I\mapsto\PP(I)\) is injective on the set of ideals of
\(\Lambda\).
\end{lemma}

\begin{proof}
Let \(S\) be the multiplicative set generated by
\(t_1,\dots,t_\ell\). Since \(\Lambda=S^{-1}P\), the
extension--contraction correspondence for localization gives
\(I=S^{-1}(I\cap P)\) for every ideal \(I\subset\Lambda\).
\end{proof}

\begin{definition}\label{def:AG}
For \(k\ge0\), define
\[
\AG_k(G,\phi)
 :=\mathcal G_P\!\left(\PP(E_k(G,\phi))\right).
\]
We call this the \emph{\(k\)th Alexander--Gr\"obner invariant} of
\((G,\phi)\). The target \(\Z^\ell\) is understood with its fixed ordered
standard basis. The reduced Gr\"obner basis of the unit ideal is \(\{1\}\),
and that of the zero ideal is the empty set. By
Theorem~\ref{thm:CF-invariance} and uniqueness of reduced Gr\"obner bases,
\(\AG_k(G,\phi)\) depends only on \((G,\phi)\).
\end{definition}

\begin{remark}\label{rem:term-order}
Throughout the paper we use the lexicographic order
\(t_1>\cdots>t_\ell\). Any other fixed monomial order also gives an
invariant.However, a different monomial order generally yields a different reduced Groebner basis, but merely a different normal form of the same contracted ideal and hence of the same Laurent ideal.
\end{remark}

Although \(\AG_k(G,\phi)\) is defined for every finitely presented group,
there is usually no canonical choice of \(\phi\). We now specialize to
oriented codimension-two links, for which the meridians provide such an
epimorphism.

\subsection{Alexander--Gr\"obner invariants for links}\label{sec:links}

Let \(d\ge1\), and let \(L=L_1\sqcup\cdots\sqcup L_\ell\subset S^{d+2}\) be an oriented
smooth spherical link, with \(L_i\cong S^d\). Fix an open regular tubular
neighborhood \(\nu(L)\), and write \(X_L:=S^{d+2}\setminus\nu(L)\).

First suppose that the components are ordered. By Alexander duality,
\(H_1(X_L;\Z)\cong\Z^\ell\), with ordered basis represented by the
positive meridians \(\mu_1,\dots,\mu_\ell\). They determine the canonical
epimorphism \(\phi\colon\pi_1(X_L)\twoheadrightarrow\Z^\ell\), defined by
\(\phi(\mu_i)=e_i\).
Let us consider the \(k\)th elementary ideal
\(E_k(L):=E_k(\pi_1(X_L),\phi)\).
In parallel, we can consider the associated Gr\"obner basis
\(\AG_k(L):=\AG_k(\pi_1(X_L),\phi)\).


For an unordered link, relabeling the components permutes the coefficient
system. We therefore retain the invariants arising from all temporary
labelings.

\begin{definition}\label{def:unordered-link-AG}
Let \(L\) be an oriented \(\ell\)-component link without a chosen ordering.
Choose a temporary labeling, with corresponding positive meridians
\(\mu_1,\dots,\mu_\ell\). For each \(\sigma\in S_\ell\), let
\(\phi_\sigma\colon\pi_1(X_L)\twoheadrightarrow\Z^\ell\) be defined by
\(\phi_\sigma(\mu_i)=e_{\sigma(i)}\). 
The \(k\)th Alexander--Gr\"obner invariant of the unordered link
is the \(S_\ell\)-indexed family
\[
\Bigl(
\AG_k(\pi_1(X_L),\phi_\sigma)
\Bigr)_{\sigma\in S_\ell},
\]
considered up to the reindexing induced by a change of the
temporary labeling. Since a
permutation of the variables need not preserve the fixed lexicographic
order, each member is reduced separately with respect to the order in
Remark~\ref{rem:term-order}.
\end{definition}
Here two oriented ordered links are called equivalent if they are
related by an orientation-preserving smooth ambient isotopy that
preserves the orientation and ordering of every component. In the
unordered case, the component ordering is not required to be preserved.
\begin{corollary}[Link invariance]\label{cor:link-invariant}
Equivalent oriented ordered links have the same Alexander--Gr\"obner
invariants in every degree. Equivalent oriented unordered links determine
the same families in Definition~\ref{def:unordered-link-AG}.
\end{corollary}

\begin{proof}
An equivalence induces an isomorphism of link groups carrying each positive
meridian to a conjugate of the corresponding positive meridian. In the
ordered case, it therefore induces an isomorphism of the pairs consisting
of the link group and its canonical coefficient system. In the unordered
case, it carries the collection of coefficient systems in
Definition~\ref{def:unordered-link-AG} to the corresponding collection and
merely reindexes the family. The assertions follow from the construction
above.
\end{proof}

\begin{example}\label{ex:k1}
Assume that \(d=1\) and \(\ell>1\). Let \(\Delta_L\) be the multivariable
Alexander polynomial, viewed in \(\Lambda\) through the natural map
\(\Z\to\K\). With the present convention for the relative Alexander
module, Torres' formula gives
\[
E_1(L)
 =\ideal{
   (t_1-1)\Delta_L,\dots,(t_\ell-1)\Delta_L
  }
 \subset\Lambda;
\]
see \cite[Theorem in \S3]{Torres}. Thus \(\Delta_L\) determines
\(E_1(L)\). Conversely, \(\Lambda\) is a UFD, so the greatest common
divisor of a finitely generated ideal is well defined up to a unit. Since
\(\ell>1\), the polynomials \(t_1-1\) and \(t_2-1\) are coprime. Thus, if
\(\Delta_L\ne0\), the greatest common divisor of \(E_1(L)\) is
\(\Delta_L\), up to a unit of \(\Lambda\); if \(\Delta_L=0\), then
\(E_1(L)=0\). Lemma~\ref{lem:contraction-injective} therefore shows that
\(\AG_1(L)\) and the image of \(\Delta_L\) in \(\Lambda\) contain the same
information, up to the usual units. 
Here, we should notice that the units are 
\(\K[t_1^{\pm1},\dots,t_\ell^{\pm1}]^\times
=\{c t_1^{a_1}\cdots t_\ell^{a_\ell}\mid c\in\K^\times,\ a_i\in\Z\}\).
We therefore focus mainly on
\(\AG_k\) for \(k>1\).
\end{example}

\subsection{Low-crossing calculations for classical links}
\label{subsec:low-crossing-examples}

Throughout this subsection, \(d=1\), \(\K=\Q\), and every link is regarded
as oriented and ordered. For each label from \emph{The Link Atlas} \cite{KnotInfo} and
\emph{KnotAtlas} \cite{KnotAtlas} appearing below, we choose an oriented planar diagram and an
ordering of its components; all reported values refer to these choices. We use
the canonical coefficient system described above, so
that \(\Phi(\mu_i)=t_i\), and the lexicographic order fixed in
Remark~\ref{rem:term-order}.

The computations were carried out by custom Python scripts using exact
rational arithmetic. Starting from Wirtinger presentations of the chosen
diagrams, we formed the Fox matrices, contracted the resulting elementary
ideals to the polynomial ring, and computed their reduced Gr\"obner bases.
Each displayed basis was checked directly against the corresponding contracted
ideal. Some elements are written in factored form for readability.
The detalied programs are put in the Home pages of the first author \cite{NHP}.

\begin{example}[Low-crossing links]\label{ex:k2}
\noindent\emph{Two components.}
For each chosen representative of a prime \(2\)-component link type in
\emph{The Link Atlas} with at most \(8\) crossings, one obtains
\(\AG_2(L)=\{1\}\). Among the chosen representatives with \(9\) crossings,
the only ones for which \(\AG_2(L)\ne\{1\}\) are
\begin{align*}
\AG_2(\mathrm{L9a32})
  &=\{t_1+1,\ t_2^2-3t_2+1\},\\
\AG_2(\mathrm{L9a33})
  &=\{t_1+t_2-1,\ t_2^2-t_2+1\},\\
\AG_2(\mathrm{L9a41})
 =\AG_2(\mathrm{L9n13})
 =\AG_2(\mathrm{L9n18})
  &=\{t_1-t_2,\ t_2^2-t_2+1\},\\
\AG_2(\mathrm{L9n19})
  &=\{t_1-t_2^3,\ t_2^4-t_2^3+t_2^2-t_2+1\}.
\end{align*}

\smallskip
\noindent\emph{Three components.}
For each chosen representative of a \(3\)-component link type with \(7\) or
\(8\) crossings, the computations give an \(\AG_2(L)\) that is neither the
empty set nor \(\{1\}\), and they give \(\AG_3(L)=\{1\}\). For many of these
representatives,
\(\AG_2(L)=\{t_1-1,t_2-1,t_3-1\}\). The exceptional values occurring among
the \(8\)-crossing representatives are listed in
Table~\ref{tab:AG2_8cross_3comp_filtered}.

\smallskip
\noindent\emph{Four components.}
For each chosen representative of a \(4\)-component link type with \(7\) or
\(8\) crossings, \(\AG_2(L)\) is again neither the empty set nor \(\{1\}\).
The corresponding bases are considerably longer and are omitted here.
\end{example}

\begin{table}[htbp]
  \centering
  \footnotesize
  \setlength{\tabcolsep}{6pt}
  \renewcommand{\arraystretch}{1.15}
  \begin{tabular}{@{}l p{0.78\linewidth}@{}}
    \toprule
    Link & \(\AG_2\) (reduced Gr\"obner basis)\\
    \midrule
    L8a16 &
    \(\{t_2-1,\ t_3-1\}\)
    \\[2pt]

    L8a19 &
    \(\{(t_1-1)(t_3^2-t_3+1),\ t_2-t_3,
       (t_3-1)(t_3^2-t_3+1)\}\)
    \\[2pt]

    L8a20 &
    \(\{(t_1-t_3)(t_1+t_3),\ (t_1-1)(t_2-1),
       (t_1+t_3)(t_3-1),\ (t_2-1)(t_2+1),
       (t_2-1)(t_3-1)\}\)
    \\[2pt]

    L8n5 &
    \(\{(t_1-1)(t_2-1),\ (t_1-1)(t_3-1),
       (t_2-1)(t_2+1),\ (t_2-1)(t_3-1),
       (t_3-1)(t_3+1)\}\)
    \\[2pt]

    L8n6 &
    \(\{(t_1-1)(t_1t_3+1),\ (t_1-1)(t_2-1),
       (t_3-1)(t_1t_3+1),\ (t_2-1)(t_2+1),
       (t_2-1)(t_3-1)\}\)
    \\
    \bottomrule
  \end{tabular}
  \caption{Exceptional values of \(\AG_2\) among the chosen
  \(8\)-crossing \(3\)-component representatives. Each value differs from
  both \(\{1\}\) and \(\{t_1-1,t_2-1,t_3-1\}\).}
  \label{tab:AG2_8cross_3comp_filtered}
\end{table}

Within the ranges examined, every chosen \(3\)- and \(4\)-component
representative has a proper, nonzero second elementary ideal, whereas all
chosen \(2\)-component representatives with at most \(8\) crossings have
\(E_2(L)=\Lambda\). This limited sample suggests that higher elementary
ideals become more varied as the number of components grows and motivates
the study of such ideals as potential refinements of the classical
multivariable Alexander polynomial.

We next record a family suggested by computer experiments.

\begin{example}[Computational experiment for \(3\)-component pretzel links]
\label{ex:3pretzel-experiment}
We record computations over \(\Q\) for the three-component pretzel links
\(L=P(2k_1,2k_2,2k_3)\), where \(k_1,k_2,k_3\) are positive integers.
To specify the orientations and component ordering, use the meridians
\(s_1,s_2,s_3\) in \cite[Proposition~2.1]{Zen}. Label and orient the
components so that \(s_i\) is a positive meridian of the \(i\)th component,
and set \(\Phi(s_i)=t_i\). Throughout this example, the polynomial ring is
\(\Q[t_1,t_2,t_3]\), with lexicographic order \(t_1>t_2>t_3\).

For clarity, the presentation used in the computations is
\[
 \pi_1(X_L)\cong
 \left\langle s_1,s_2,s_3\ \middle|\ v_1u_2,\ v_2u_3\right\rangle,
\]
where indices are read modulo \(3\) and
\[
 \begin{aligned}
 h_i&=s_{i+1}^{-1}s_i,\\
 u_i&=h_i^{-k_i}s_i h_i^{k_i},
 &\qquad
 v_i&=h_i^{-k_i}s_{i+1}^{-1}h_i^{k_i}.
 \end{aligned}
\]
This is the all-even specialization of the presentation in
\cite[Proposition~2.1]{Zen}, with one redundant relator omitted.
The second elementary ideal is generated by the entries of the resulting
\(2\times3\) Fox matrix. After clearing Laurent monomial denominators,
we saturate the resulting polynomial ideal by \(t_1t_2t_3\) and compute
its reduced Gr\"obner basis. All calculations below use exact rational
arithmetic.

For every pair \(1\le m,n\le9\), the computation gives
\[
 \AG_2\bigl(P(2,2m,2n)\bigr)
 =\left\{t_1-1,\ t_2-1,\ t_3^{\gcdop(m,n)}-1\right\}.
\]
This formula is reported here as a finite computational observation;
no assertion for arbitrary \(m,n\) is made.

Further computations give
\[
 \AG_2\bigl(P(4,4,6)\bigr)
 =\left\{t_1-1,\ t_2^2-1,\ t_3-1\right\},
\]
and
\[
 \begin{aligned}
 \AG_2\bigl(P(4,6,12)\bigr)
 =\bigl\{\,&t_1^2-t_2^2,\ (t_1+t_2)(t_2-1),\\
          &(t_1+t_2)(t_3^3-1),\ t_2^3-t_3^3,\\
          &(t_3-1)(t_2^2+t_2t_3+t_3^2)\,\bigr\}.
 \end{aligned}
\]
Some elements in the latter basis are factored for readability.
With precisely the same meridian convention and monomial order, the
computed cardinalities are
\[
 \begin{array}{c|cccc}
 L & P(4,4,6) & P(4,6,12) & P(6,6,6) & P(10,10,10)\\ \hline
 |\AG_2(L)| & 3 & 5 & 7 & 7
 \end{array}
\]
Thus both the cardinality and the algebraic form of the reduced basis
vary within this three-component family.

These data concern oriented ordered links with the conventions specified
above. For the corresponding oriented unordered links, one recomputes the
reduced basis for every ordering of the components, as prescribed in
Definition~\ref{def:unordered-link-AG}. The present example is included
as computational evidence and is not used in the proofs below.
\end{example}

These computations illustrate both the range of behavior and the
computational complexity of the invariant. They motivate the infinite
families with explicit formulas studied in Section~\ref{sec:families}.

%

\begin{example}[The same Alexander polynomial but different \(\AG_2\) values]
\label{ex:AG-refines-Alexander}
Extending the same computation to the chosen two- and three-component
representatives with at most \(10\) crossings gives links whose
multivariable Alexander polynomials agree, although their second
Alexander--Gr\"obner invariants do not. We record three representative
pairs. All statements below use the orientations, component orderings,
coefficient field, and monomial order fixed at the beginning of this
subsection.

\begin{enumerate}[(i)]
\item Let \(B=\mathrm{L6a4}\) and \(N=\mathrm{L8n5}\),
where \(B\) denotes the Borromean rings. Their multivariable Alexander
polynomials satisfy
\[
 \Delta_B\doteq\Delta_N
 \doteq (t_1-1)(t_2-1)(t_3-1),
\]
where \(\doteq\) denotes equality up to the usual Laurent units. On the
other hand,
\[
 \AG_2(B)
 =\bigl\{
 (t_1-1)(t_2-1),\
 (t_1-1)(t_3-1),\
 (t_2-1)(t_3-1)
 \bigr\},
\]

\begin{align*}
 \AG_2(N)
 =\bigl\{&
 (t_1-1)(t_2-1),\
 (t_1-1)(t_3-1),\
 (t_2-1)(t_2+1),\\
 & (t_2-1)(t_3-1),\
 (t_3-1)(t_3+1)
 \bigr\}.
\end{align*}
Thus \(\AG_2(B)\ne\AG_2(N)\).

\item The two-component links \(\mathrm{L9a4}\) and \(\mathrm{L10n35}\)
have the same multivariable Alexander polynomial; an integral representative
is \(2(t_1-1)(t_2-1)(t_2^2-t_2+1)\).  Nevertheless,
\[
 \AG_2(\mathrm{L9a4})=\{1\},
 \qquad
 \AG_2(\mathrm{L10n35})
 =\{t_1-1,\ t_2^2-t_2+1\}.
\]
Hence, in this pair, the second elementary ideal is the unit ideal for one
link and a proper ideal for the other.

\item The two-component links \(\mathrm{L10n56}\) and \(\mathrm{L10n57}\)
have the common multivariable Alexander polynomial
\[
 \Delta_{\mathrm{L10n56}}\doteq\Delta_{\mathrm{L10n57}}
 \doteq (t_1-1)(t_1+1)^2(t_2-1),
\]

\[
 \AG_2(\mathrm{L10n56})
 =\left\{t_1+1,\ t_2^2-\frac32t_2+1\right\}, \qquad 
 \AG_2(\mathrm{L10n57})
 =\left\{t_1+1,\ t_2^2-\frac52t_2+1\right\}.
\]
Thus the invariant can also distinguish two links for which both second
Alexander--Gr\"obner invariants are nontrivial.
\end{enumerate}

These examples show directly that \(\AG_2(L)\) is not determined by the
multivariable Alexander polynomial. The three pairs were chosen to display,
respectively, a three-component example, the distinction between a unit and
a proper second elementary ideal, and a distinction between two proper
second elementary ideals.
\end{example}

\begin{remark}[Further low-crossing pairs]
\label{rem:further-Alexander-collisions}
In the same data set, fourteen further pairs have multivariable Alexander
polynomials agreeing up to the usual units but different second
Alexander--Gr\"obner invariants:
\begin{itemize}
\item
\((\mathrm{L6a4},\mathrm{L10n70})\), together with
\((\mathrm{L9n25},L)\) for
\(L\in\{\mathrm{L8n5},\mathrm{L10n70}\}\);

\item
\((\mathrm{L8a16},\mathrm{L10n79})\),
\((\mathrm{L9n28},\mathrm{L10n83})\),
\((\mathrm{L9a10},\mathrm{L10n35})\), and
\((\mathrm{L9a14},\mathrm{L10n39})\);

\item
\((L,\mathrm{L10a32})\) for
\(L\in\{\mathrm{L10a3},\mathrm{L10a7},
        \mathrm{L10a34},\mathrm{L10a36}\}\);

\item
\((L,\mathrm{L10a39})\) for
\(L\in\{\mathrm{L10a28},\mathrm{L10a29}\}\);

\item
\((\mathrm{L10n32},\mathrm{L10n36})\).
\end{itemize}
For the last pair the common Alexander polynomial is zero; for all the other
pairs in this remark it is nonzero. Together with the three pairs in
Example~\ref{ex:AG-refines-Alexander}, this gives seventeen pairs among
the chosen two- and three-component representatives with at most \(10\)
crossings. Since the purpose here is only to exhibit information not
detected by the Alexander polynomial, we omit the Fox matrices and the full
reduced Gr\"obner bases for these additional pairs; those data are more
naturally placed in an appendix or a supplementary table.
\end{remark}

\section{Explicit infinite families}\label{sec:families}

We now turn to two infinite families for which the Alexander--Gr\"obner invariants admit explicit formulas.
These families are illustrated in Figure~\ref{fig:torus-pretzel}.

\begin{figure}[htbp]
\centering
\begin{picture}(100,125)
\IfFileExists{pretzelFigure3.pdf}{%
  \put(94,54){\pcc{pretzelFigure3}{0.5}}%
}{\put(110,63){\fbox{\texttt{Missing figure: 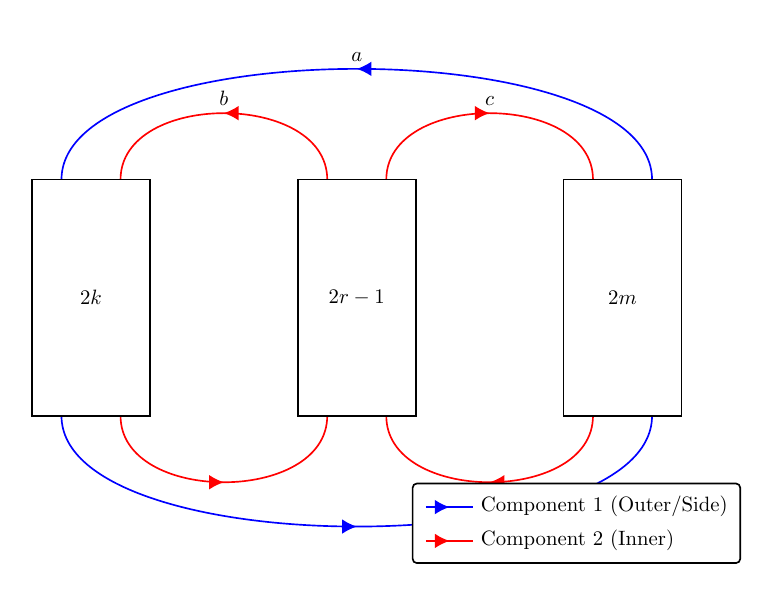}}}}
\IfFileExists{mutorus4.eps}{%
  \put(-85,55){\pc{mutorus4}{0.31156688125}}%
}{\put(-70,62){\fbox{\texttt{Missing figure: 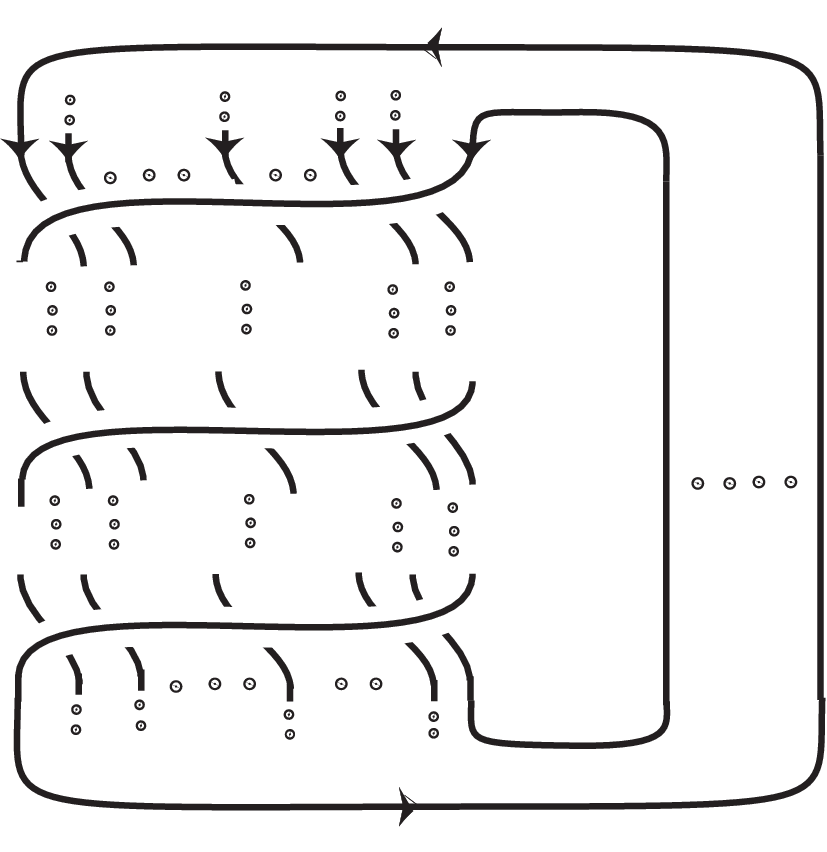}}}}
\end{picture}
\caption{Left: the torus link $T(\ell n,\ell m)$, which has $\ell$ components. Right: the three-strand pretzel link $P(2k,2r-1,2m)$, which has two components.}
\label{fig:torus-pretzel}
\end{figure}

\subsection{Torus links $T(\ell n,\ell m)$}\label{subsec:torus-links}

Fix an integer \(\ell>1\) and coprime positive integers \(m,n\), and let
\(L:=T(\ell n,\ell m)\subset S^3\) be the oriented ordered
\(\ell\)-component torus link shown on the left in
Figure~\ref{fig:torus-pretzel}.  Let \(\mu_i\) be the positive meridian of
its \(i\)th component, and use the canonical epimorphism
\(\phi\colon\pi_1(X_L)\twoheadrightarrow\Z^\ell\) given by
\(\phi(\mu_i)=e_i\).  Put
\[
\begin{gathered}
 P:=\Q[t_1,\dots,t_\ell],\qquad
 \Lambda:=\Q[t_1^{\pm1},\dots,t_\ell^{\pm1}],\qquad
 U:=t_1\cdots t_\ell,\\
 F_i:=t_1\cdots t_i\quad(1\le i<\ell),\qquad
 \mathfrak a:=(t_1-1,\dots,t_\ell-1)\subset P,
\end{gathered}
\]
and
\[
 D(U):=U^{mn}-1,\qquad
 A(U):=\frac{D(U)}{U^m-1},\qquad
 B(U):=\frac{D(U)}{U^n-1}.
\]
Set \(\Theta(U):=A(U)B(U)\).

By \cite{AK}, the link group admits the deficiency-one presentation
\footnote{The authors \cite{AK} give a presentation containing a generator \(f_0\)
and the relation \(f_0=1\). Eliminating \(f_0\) yields the following
deficiency-one presentation.}
\[
 \pi_1(X_L)\cong
 \Bigl\langle \alpha,\beta,f_1,\dots,f_{\ell-1}
 \ \Bigm|\
 \alpha^n\beta^{-m},\
 \alpha^n f_i \beta^{-m} f_i^{-1}\ (1\le i<\ell)
 \Bigr\rangle.
\]
Then, the abelianization is given by
 \[ [\alpha]=m(e_1+ \cdots +e_{\ell}), \quad 
[\beta]=n(e_1+ \cdots +e_{\ell}), \quad 
[f_i]= e_1+ \cdots +e_{i}.
\]
Thus, the induced group-ring map satisfies
\[
\Phi(\alpha)=U^m,\qquad \Phi(\beta)=U^n,\qquad \Phi(f_i)=F_i.
\]
After Fox differentiation and subtraction of the first row from every
other row, the Fox matrix is row-equivalent over \(P\) to
\[
 M_\ell=
 \begin{pmatrix}
 A(U) & -B(U) & 0 & 0 & \cdots & 0\\
 0 & -(F_1-1)B(U) & D(U) & 0 & \cdots & 0\\
 0 & -(F_2-1)B(U) & 0 & D(U) & \cdots & 0\\
 \vdots & \vdots & \vdots & \vdots & \ddots & \vdots\\
 0 & -(F_{\ell-1}-1)B(U) & 0 & 0 & \cdots & D(U)
 \end{pmatrix}
 \in\Mat_{\ell\times(\ell+1)}(P).
\]

\begin{lemma}\label{lem:toruslink-saturation}
If \(h(U)\in\Q[U]\) satisfies \(h(0)\ne0\), then
\[
 (h(U)\Lambda)\cap P=(h(U)),\qquad
 (h(U)\mathfrak a\Lambda)\cap P=h(U)\mathfrak a.
\]
\end{lemma}

\begin{proof}
It is enough to show that the two ideals on the right are saturated with
respect to every variable \(t_i\).  Since
\(h(U)\equiv h(0)\pmod{t_i}\), one has
\((h(U),t_i)=P\), and hence \((h(U)):t_i=(h(U))\).  If
\(t_i f\in h(U)\mathfrak a\), write \(f=h(U)g\) by the preceding colon
identity.  Cancelling \(h(U)\) gives \(t_i g\in\mathfrak a\); since
\(t_i\equiv1\pmod{\mathfrak a}\), this implies \(g\in\mathfrak a\).
Thus \((h(U)\mathfrak a):t_i=h(U)\mathfrak a\), and localization at the
variables gives both contraction formulas.
\end{proof}

\begin{proposition}[Elementary ideals of torus links]
\label{prop:toruslink-Ek}
Define the monic polynomials
\[
\begin{aligned}
 P_k(U)&:=D(U)^{\ell-1-k}\Theta(U) &&(1\le k<\ell),\\
 P_\ell(U)&:=\gcdop\bigl(A(U),B(U)\bigr)
 =\frac{D(U)(U-1)}{(U^m-1)(U^n-1)}.
\end{aligned}
\]
Then
\[
 E_k(L)=P_k(U)\mathfrak a\Lambda,
 \qquad
 \PP(E_k(L))=P_k(U)\mathfrak a
 \qquad(1\le k<\ell),
\]
whereas
\[
 E_\ell(L)=\bigl(P_\ell(U)\bigr)\subset\Lambda,
 \qquad
 \PP(E_\ell(L))=\bigl(P_\ell(U)\bigr)\subset P.
\]
\end{proposition}

\begin{proof}
Write \(A=A(U)\), \(B=B(U)\), and \(D=D(U)\), and put
\(J_r:=I_r(M_\ell)\).  Since row operations preserve determinantal ideals and the presentation has
\(\ell+1\) generators,
\(E_k(L)=J_{\ell+1-k}\Lambda\) for \(1\le k\le\ell\).

The entries of \(M_\ell\) give
\(J_1=(A,B,D)=(A,B)\), because
\(D=(U^m-1)A=(U^n-1)B\).  Since
\(\gcdop(U^m-1,U^n-1)=U-1\), B\'ezout's identity in \(\Q[U]\) yields
\[
 J_1=\bigl(\gcdop(A,B)\bigr)
 =\left(\frac{D\,(U-1)}{(U^m-1)(U^n-1)}\right)
 =\bigl(P_\ell(U)\bigr).
\]

The nonzero \(2\times2\) minors involving the first row generate
\[
 \bigl(\Theta(U)(F_1-1),\dots,\Theta(U)(F_{\ell-1}-1),AD,BD\bigr).
\]
Every remaining nonzero \(2\times2\) minor is of the form
\(\pm BD(F_i-1)\) or \(\pm D^2\), and hence lies in this ideal.
Using
\(AD=\Theta(U)(U^n-1)\) and \(BD=\Theta(U)(U^m-1)\), we obtain
\[
 J_2=\Theta(U)\bigl(F_1-1,\dots,F_{\ell-1}-1,U^m-1,U^n-1\bigr)
     =\Theta(U)\mathfrak a.
\]
Here the last equality follows from
\((U^m-1,U^n-1)=(U-1)\) and
\[
\begin{aligned}
 F_1-1&=t_1-1,\qquad
 t_i-1=(F_i-1)-t_i(F_{i-1}-1) &&(2\le i<\ell),\\
 t_\ell-1&=(U-1)-t_\ell(F_{\ell-1}-1).
\end{aligned}
\]

For \(2\le r\le\ell\), every nonzero \(r\times r\) minor uses at least
\(r-2\) of the last \(\ell-1\) columns.  Each such column has the single
nonzero entry \(D\), so Laplace expansion gives
\(J_r\subset D^{r-2}J_2\).  Conversely, each generator in the displayed
first-row list for \(J_2\) can be enlarged by \(r-2\) unused diagonal
row--column pairs.  Hence
\[
 J_r=D^{r-2}J_2=D^{r-2}\Theta(U)\mathfrak a
 \qquad(2\le r\le\ell).
\]
It follows that the asserted formulas hold after extension to
\(\Lambda\).  Finally,
\(P_k(0)=(-1)^{\ell-1-k}\ne0\) for \(k<\ell\), while
\(P_\ell(0)=1\).  Lemma~\ref{lem:toruslink-saturation} gives the stated
contractions to \(P\).
\end{proof}

\begin{corollary}[Reduced Gr\"obner bases]
\label{cor:toruslink-groebner}
With respect to the lexicographic order \(t_1>\cdots>t_\ell\),
\[
 \AG_k(L)=
 \begin{cases}
  \bigl\{P_k(U)(t_i-1)\mid 1\le i\le\ell\bigr\},&1\le k<\ell,\\
  \{P_\ell(U)\},&k=\ell.
 \end{cases}
\]
\end{corollary}

\begin{proof}
The assertion for \(k=\ell\) is immediate because \(P_\ell(U)\) is
monic.  Fix \(k<\ell\), and put \(d:=\deg_U P_k\).  Since
\(\{t_1-1,\dots,t_\ell-1\}\) is the reduced Gr\"obner basis of
\(\mathfrak a\) and
\(\LM(P_k(U)f)=U^d\LM(f)\) for every nonzero \(f\in P\),
\[
 \inop\bigl(P_k(U)\mathfrak a\bigr)
 =U^d\inop(\mathfrak a)
 =(U^d t_1,\dots,U^d t_\ell).
\]
Thus the displayed generators form a Gr\"obner basis.  They are monic,
and every nonleading monomial of \(P_k(U)(t_i-1)\) is of the form
\(U^d\), \(U^s\), or \(U^s t_i\) with \(s<d\); none is divisible by any
\(U^d t_j\).  Hence the basis is reduced.
\end{proof}

\subsection{Two-component pretzel links}\label{subsec:pretzel-two}

We consider the oriented ordered link
\[
L=P(2k,2r-1,2m)
\qquad (k,m,r\ge 1).
\]
We use \(r\) for the middle twist parameter so that \(\ell\) continues to
denote the number of link components.  The argument has three steps: we
first reduce the second elementary ideal to a three-generated ideal in
\(\Q[x,y]\), then determine its cyclotomic fibers, and finally assemble the
local data into a lexicographic Gr\"obner basis.

The fundamental group has the presentation
\[
\pi_1(S^3\setminus \nu(L))
 =\langle a,b,c\mid r_1,r_2\rangle,
\]
where
\[
\begin{aligned}
r_1&=(cb)^r b^{-1}(cb)^{-r}(ac^{-1})^m c(ca^{-1})^m,\\
r_2&=(ab)^{-k}b(ab)^k(cb)^{r-1}c^{-1}(cb)^{-(r-1)}.
\end{aligned}
\]
This presentation is obtained from the Wirtinger presentation of Figure~\ref{fig:torus-pretzel},
or from
\cite[Proposition~2.1]{Zen} by taking
\[
(p_1,p_2,p_3)=(2k,2r-1,2m),
\qquad
(s_1,s_2,s_3)=(a,b^{-1},c).
\]
The generators \(a,b,c\) are the oriented meridians shown in
Figure~\ref{fig:torus-pretzel}, with \(a\) on the first (outer) component and
\(b,c\) on the second (inner) component.  Since \(\phi\) sends each meridian to
its linking numbers with the components, reading these off the figure gives
\(\phi(a)=e_1\) and \(\phi(b)=\phi(c)=e_2\).
Thus the induced group-ring homomorphism sends
\[
a\longmapsto t_1=:y,
\qquad
b,c\longmapsto t_2=:x,
\]
so that the fixed order \(t_1>t_2\) becomes \(\lex(y>x)\).  Set
\[
R:=\Q[x,y],
\qquad
\Lambda:=\Q[x^{\pm1},y^{\pm1}],
\qquad
q:=2r-1,
\]
and define
\[
A_q(x):=\frac{1+x^q}{1+x},
\qquad
G_k(x,y):=\frac{1-(xy)^k}{1-xy},
\qquad
H_m(x,y):=\frac{x^m-y^m}{x-y}.
\]
These are polynomials: \(q\) is odd, and the latter two expressions are
finite geometric sums.

\begin{proposition}[Reduction to a three-generated ideal]
\label{prop:pretzel-reduction}
Let \(I:=\PP(E_2(L))\subset R\), where we use the identification
\(t_1=y\), \(t_2=x\).  Then
\[
I=\ideal{A_q(x),G_k(x,y),H_m(x,y)}.
\]
Consequently,
\[
\AG_2(L)=\mathcal G_{\lex(y>x)}(I),
\]
where \(\mathcal G_{\lex(y>x)}(I)\) denotes the reduced Gr\"obner basis of
\(I\subset R\) for the order \(\lex(y>x)\).  Moreover, in \(R/I\) one has
\((xy)^k=1\) and \(x^m=y^m\).
\end{proposition}

\begin{proof}
Let
\[
 A_\phi
 =\left(
   \Phi\!\left(\frac{\partial r_i}{\partial z_j}\right)
  \right)
 \in\Mat_{2\times 3}(\Lambda),
 \qquad (z_1,z_2,z_3)=(a,b,c),
\]
be the Fox matrix of the displayed presentation. 
Applying the Fox rules to \(r_1,r_2\), and using
\(\Phi(a)=y\) and \(\Phi(b)=\Phi(c)=x\), gives
\[
\begin{aligned}
 (A_\phi)_{11}
 &=x^{-1}(1-x)
   \sum_{i:\,0\le i<m}(yx^{-1})^i
  =x^{-m}(1-x)H_m,\\
 (A_\phi)_{12}
 &=(x-1)\sum_{i:\,0\le i<r}x^{2i}-x^q
  =-A_q,\\
 (A_\phi)_{21}
 &=(xy)^{-k}(x-1)
   \sum_{i:\,0\le i<k}(xy)^i
  =(xy)^{-k}(x-1)G_k,\\
 (A_\phi)_{23}
 &=(x-1)\sum_{i:\,0\le i<r-1}x^{2i}-x^{q-1}
  =-A_q,
\end{aligned}
\]
where the last sum is understood to be zero when \(r=1\).  Since
\(\Phi(r_i)=1\), the Fox fundamental identity yields
\[
 (y-1)(A_\phi)_{i1}
 +(x-1)(A_\phi)_{i2}
 +(x-1)(A_\phi)_{i3}=0
 \qquad (i=1,2).
\]
It follows that
\[
 A_\phi=
 \begin{pmatrix}
 x^{-m}(1-x)H_m
 &-A_q
 &x^{-m}\bigl(x^mA_q+(y-1)H_m\bigr)
 \\
 (xy)^{-k}(x-1)G_k
 &(xy)^{-k}\bigl((1-y)G_k+(xy)^kA_q\bigr)
 &-A_q
 \end{pmatrix}.
\]
Since the presentation has three generators,
\(E_2(L)=I_1(A_\phi)\). Multiplying individual generators of an ideal by units does not change the ideal.  Therefore
\[
 E_2(L)=I_1(A_\phi)
 =\ideal{
 A_q,\ (1-x)H_m,\ x^mA_q+(y-1)H_m,\
 (x-1)G_k,\ (1-y)G_k+(xy)^kA_q
 }_{\Lambda}.
\]

Put \(J:=\ideal{A_q,G_k,H_m}\subset R\).  The displayed generators lie
in \(J\Lambda\), so \(E_2(L)\subset J\Lambda\).  Conversely, the same
generators show that
\(A_q,(x-1)G_k,(x-1)H_m\in E_2(L)\).  Since \(A_q(1)=1\), one has
\((A_q,x-1)=R\).  Multiplying a B\'ezout relation first by \(G_k\) and then
by \(H_m\) gives \(G_k,H_m\in E_2(L)\).  Hence
\(E_2(L)=J\Lambda\).

Now \(A_q(x)\equiv1\pmod{x}\) and \(G_k(x,y)\equiv1\pmod{y}\).  Thus
\((J,x)=(J,y)=R\), so the classes of \(x\) and \(y\) are units in
\(R/J\).  Equivalently,
\(J:(xy)^\infty=J\), and localization--contraction gives
\(I=(J\Lambda)\cap R=J\).  Finally,
\((1-xy)G_k=1-(xy)^k\) and \((x-y)H_m=x^m-y^m\), which prove the
last assertion.
\end{proof}
If \(q=1\), then \(A_q=1\), and
Proposition~\ref{prop:pretzel-reduction} gives \(I=R\) and
\(\AG_2(L)=\{1\}\).  Henceforth assume \(q>1\).

Let \(\Phi_n(x)\) denote the \(n\)-th cyclotomic polynomial.  For every
divisor \(\delta\mid q\) with \(\delta>1\), put
\(\mathcal K_\delta:=\Q[x]/(\Phi_{2\delta}(x))\).Since $\Phi_{2\delta}$ is irreducible over $\mathbb{Q}$, $K_\delta$ is the cyclotomic field; in particular $K_\delta[y]$ is a PID.Denote the class
of \(x\) by \(x_\delta\).  For the root calculation
below, choose an embedding of \(\mathcal K_\delta\) into
\(\overline{\Q}\); then \(x_\delta\) is viewed as a primitive
\(2\delta\)-th root of unity.  The resulting formulas are independent of
this choice.

Set \(g:=\gcdop(m,k)\), choose \(u,v\in\Z\) with \(um+vk=g\), and
put \(e:=um-vk\).
Since \((xy)^k=1\), the classes of \(x\) and \(y\) are units in
\(R/I\), and \(y^k=x^{-k}\).  Together with \(x^m=y^m\), this gives
\[
y^g=(y^m)^u(y^k)^v=x^{um-vk}=x^e
\qquad\text{in }R/I.
\]
Any other B\'ezout pair changes \(e\) by a multiple of
\(2\lcm(m,k)\).  Hence \(x_\delta^e\) is independent of the chosen pair
whenever \(\delta\mid\lcm(m,k)\).  For a condition \(\mathsf C\), write
\(\mathbf 1_{\mathsf C}\in\{0,1\}\) for its indicator.

For every divisor \(\delta\mid q\) with \(\delta>1\), let
\[
p_\delta(y):=
\gcdop\bigl(G_k(x_\delta,y),H_m(x_\delta,y)\bigr)
\in\mathcal K_\delta[y]
\]
be chosen monic, and set \(d_\delta:=\deg_y p_\delta\).

\begin{lemma}[Cyclotomic fibers]\label{lem:pretzel-local-uniform}
For every divisor \(\delta\mid q\) with \(\delta>1\),
\begin{equation}\label{eq:pretzel-local-factor}
p_\delta(y)=
\begin{cases}
1,
& \delta\nmid\lcm(m,k),
\\[6pt]
\displaystyle
\frac{y^g-x_\delta^e}
 {(y-x_\delta)^{\mathbf 1_{\delta\mid k}}
  (y-x_\delta^{-1})^{\mathbf 1_{\delta\mid m}}},
& \delta\mid\lcm(m,k).
\end{cases}
\end{equation}
In particular,
\begin{equation}\label{eq:pretzel-local-degree}
d_\delta=
\begin{cases}
0,
& \delta\nmid\lcm(m,k),
\\[3pt]
g-\mathbf 1_{\delta\mid m}-\mathbf 1_{\delta\mid k},
& \delta\mid\lcm(m,k).
\end{cases}
\end{equation}

Let
\[
\mathcal D:=
\{\delta>1:\,\delta\mid q,\ d_\delta>0\}.
\]
If \(\mathcal D=\varnothing\), then \(I=R\).  If
\(\mathcal D\ne\varnothing\), then
\[
R/I\cong
\prod_{\delta:\,\delta\in\mathcal D}
\mathcal K_\delta[y]/(p_\delta(y)),
\qquad
\dimQ(R/I)=
\sum_{\delta:\,\delta\in\mathcal D}d_\delta\deg\Phi_{2\delta}.
\]
\end{lemma}

\begin{proof}
Write \(\mu_n:=\{\zeta\in\overline{\Q}^{\,\times}:\,\zeta^n=1\}\).
The roots in \(y\) of the two specialized polynomials are simple and are
given by
\[
H_m(x_\delta,y)=0
\iff y=x_\delta\sigma
\quad(\sigma\in\mu_m\setminus\{1\}),
\]
\[
G_k(x_\delta,y)=0
\iff y=x_\delta^{-1}\tau
\quad(\tau\in\mu_k\setminus\{1\}).
\]
The simultaneous solutions of \(y^m=x_\delta^m\) and
\(y^k=x_\delta^{-k}\) are \(y=x_\delta\sigma\), where
\(\sigma\in\mu_m\cap x_\delta^{-2}\mu_k\).
To obtain the common roots of \(H_m(x_\delta,y)\) and
\(G_k(x_\delta,y)\), one must exclude
\(\sigma=1\) and \(\sigma=x_\delta^{-2}\), respectively.
This intersection is nonempty precisely when
\(x_\delta^{-2}\in\mu_m\mu_k=\mu_{\lcm(m,k)}\), or equivalently when
\(\delta\mid\lcm(m,k)\).  When nonempty, it is a
coset of \(\mu_m\cap\mu_k=\mu_g\), and hence has \(g\) elements.
For every \(\sigma\) in this coset ,the corresponding value \(y=x_\delta\sigma\) satisfies
\[
y^m=x_\delta^m,
\qquad
y^k=x_\delta^{-k},
\qquad
y^g=(y^m)^u(y^k)^v=x_\delta^e.
\]
Since \(y^g-x_\delta^e\) has exactly \(g\) simple roots, these \(g\)
values are precisely its roots.  The excluded root \(y=x_\delta\) occurs
precisely when \(\delta\mid k\), while the excluded root
\(y=x_\delta^{-1}\) occurs precisely when \(\delta\mid m\).  These roots
are distinct because \(\delta>1\).  This proves
\eqref{eq:pretzel-local-factor}, and
\eqref{eq:pretzel-local-degree} follows.

Since \(q\) is odd,
\[
A_q(x)=\prod_{\delta:\,\delta\mid q,\ \delta>1}\Phi_{2\delta}(x).
\]
The Chinese remainder theorem and
Proposition~\ref{prop:pretzel-reduction} therefore give
\[
R/I\cong
\prod_{\delta:\,\delta\mid q,\ \delta>1}
\frac{\mathcal K_\delta[y]}{(G_k(x_\delta,y),H_m(x_\delta,y))}.
\]
Because \(\mathcal K_\delta[y]\) is a PID, the ideal in the denominator
is \((p_\delta)\).  If \(\mathcal D=\varnothing\), every factor is the
zero ring, so \(I=R\).  Otherwise the zero factors, characterized by
\(p_\delta=1\), may be omitted.  Finally,
\(\dimQ(\mathcal K_\delta[y]/(p_\delta))
=d_\delta[\mathcal K_\delta:\Q]=d_\delta\deg\Phi_{2\delta}\), which
proves the dimension formula.
\end{proof}
For every \(\delta\in\mathcal D\), the polynomial \(p_\delta(y)\)
divides \(y^g-x_\delta^e\). Since \(x_\delta^e\ne0\) and
\(\mathcal K_\delta\) has characteristic zero, this binomial
is relatively prime to its derivative \(g y^{g-1}\).
Thus \(p_\delta\) is square-free.
Since each \(\mathcal K_\delta\) is a number field, the preceding
decomposition shows that, when \(I\ne R\), the quotient \(R/I\)
is a finite product of number fields, equivalently a finite
\'etale \(\Q\)-algebra.
In particular, \(I\) is radical; this also holds when \(I=R\).

For the uniform construction, assume \(\mathcal D\ne\varnothing\).  For
\(0\le n\le g\), define
\[
\Pi_n(x):=
\prod_{\delta:\,\delta\in\mathcal D,\ d_\delta>n}\Phi_{2\delta}(x),
\qquad
\mathcal S:=\{0,g\}\cup\{d_\delta:\,\delta\in\mathcal D\},
\]
where an empty product is \(1\).  In particular, \(\Pi_g=1\), while
\(\Pi_0\) is the product of the cyclotomic factors indexed by
\(\mathcal D\).  Every \(\delta\in\mathcal D\) divides
\(\lcm(m,k)\), and hence every factor \(\Phi_{2\delta}\) of \(\Pi_0\)
divides \(x^{2\lcm(m,k)}-1\).  Since different B\'ezout pairs change
\(e\) by a multiple of \(2\lcm(m,k)\), the residue class of \(x^e\)
modulo \(\Pi_0\) is independent of the chosen pair.  Since
\(\Pi_0(0)=1\), the class
of \(x\) is invertible in \(\Q[x]/(\Pi_0)\).  Let \(h(x)\) be the unique
polynomial satisfying
\[
\deg h<\deg\Pi_0,
\qquad
h\equiv x^e\pmod{\Pi_0},
\]
where a negative power is interpreted in \(\Q[x]/(\Pi_0)\).

For \(j\in\mathcal S\setminus\{0,g\}\), one has
\[
\frac{\Pi_0}{\Pi_j}
=\prod_{\delta:\,\delta\in\mathcal D,\ d_\delta\le j}\Phi_{2\delta}.
\]
This product is nonconstant because \(j=d_\delta\) for some
\(\delta\in\mathcal D\).  The Chinese remainder theorem, applied
coefficientwise, gives a unique monic polynomial
\(Q_j(x,y)\in\Q[x,y]\) of
\(y\)-degree \(j\) such that each coefficient has \(x\)-degree less than
\(\deg(\Pi_0/\Pi_j)\) and
\begin{equation}\label{eq:pretzel-Qj}
Q_j(x_\delta,y)=y^{j-d_\delta}p_\delta(y)
\qquad
(\delta\in\mathcal D,\ d_\delta\le j).
\end{equation}

\begin{theorem}[Uniform Gr\"obner-basis construction]
\label{thm:pretzel-uniform}
Define
\[
\mathcal B:=
\{\Pi_0(x),\ y^g-h(x)\}
\cup
\{\Pi_j(x)Q_j(x,y):\,j\in\mathcal S\setminus\{0,g\}\}.
\]
Then \(\mathcal B\) is a Gr\"obner basis of \(I\) with respect to
\(\lex(y>x)\), and
\[
\inop_{\lex}(I)=
\ideal{x^{\deg\Pi_j}y^j:\,j\in\mathcal S}.
\]
Consequently, \(\AG_2(L)\) is the reduced Gr\"obner basis obtained
by monic normalization and interreduction of \(\mathcal B\).
The displayed at-most-four-element Gr\"obner basis may contain redundant
elements before this interreduction.
\end{theorem}

\begin{proof}
Put \(J:=\ideal{\mathcal B}\).  For every
\(\delta\in\mathcal D\),
Lemma~\ref{lem:pretzel-local-uniform} gives
\(p_\delta(y)\mid y^g-x_\delta^e\).
Since \(\Pi_0(x_\delta)=0\) and \(h(x_\delta)=x_\delta^e\), the first two
elements of \(\mathcal B\) belong to \(I\).  For an intermediate index
\(j\), either \(d_\delta>j\), in which case
\(\Phi_{2\delta}\mid\Pi_j\), or \(d_\delta\le j\), in which case
\eqref{eq:pretzel-Qj} makes \(Q_j(x_\delta,y)\) a multiple of
\(p_\delta(y)\).  The local decomposition therefore gives
\(J\subset I\).

The leading monomials are
\[
\LM(\Pi_0)=x^{\deg\Pi_0},
\qquad
\LM(y^g-h)=y^g,
\qquad
\LM(\Pi_j Q_j)=x^{\deg\Pi_j}y^j.
\]
Hence, with \(M:=\ideal{x^{\deg\Pi_j}y^j:\,j\in\mathcal S}\), one has
\(M\subset\inop_{\lex}(J)\).  For every
\(0\le n\le g\), the value of \(\Pi_n\) agrees with \(\Pi_j\), where
\(j\) is the largest element of \(\mathcal S\) not exceeding \(n\).
Thus adjoining \(x^{\deg\Pi_n}y^n\) for all \(0\le n\le g\) does not
change \(M\).  Since \(\Pi_g=1\), the standard monomials of \(R/M\) are
precisely
\[
x^a y^n
\qquad
(0\le n<g,\ 0\le a<\deg\Pi_n).
\]
Consequently,
\[
\begin{aligned}
\dim_{\mathbb Q}(R/M)
&=\sum_{n:\,0\le n<g}\deg\Pi_n \\
&=\sum_{n:\,0\le n<g}
\sum_{\substack{\delta\in\mathcal D\\ d_\delta>n}}
\deg\Phi_{2\delta} \\
&=\sum_{\delta\in\mathcal D}
\bigl|\{n:0\le n<g,\ n<d_\delta\}\bigr|
\deg\Phi_{2\delta} \\
&=\sum_{\delta\in\mathcal D}
d_\delta\deg\Phi_{2\delta} \\
&=\dim_{\mathbb Q}(R/I),
\end{aligned}
\]
where the last equality is the dimension formula in
Lemma~\ref{lem:pretzel-local-uniform}.

Since \(M\subseteq\inop_{\lex}(J)\) and
\(R/M\) is finite-dimensional over \(\mathbb Q\),
the quotient \(R/\inop_{\lex}(J)\) is also
finite-dimensional.  The standard-monomial theorem gives
\(\dimQ(R/J)=\dimQ(R/\inop_{\lex}(J))\).
The inclusions \(J\subseteq I\) and
\(M\subseteq\inop_{\lex}(J)\) give exact sequences
\[
0\longrightarrow I/J
\longrightarrow R/J
\longrightarrow R/I
\longrightarrow 0, 
\]
\[
0\longrightarrow \inop_{\lex}(J)/M
\longrightarrow R/M
\longrightarrow R/\inop_{\lex}(J)
\longrightarrow 0.
\]
Consequently,
\[
\dimQ(R/I)
\le
\dimQ(R/J)
=
\dimQ(R/\inop_{\lex}(J))
\le
\dimQ(R/M).
\]
The preceding computation gives
\(\dimQ(R/I)=\dimQ(R/M)\), so equality holds throughout.  The two kernels in the
displayed exact sequences therefore have dimension zero.
Thus, we obtain \(I/J=0\) and \(\inop_{\lex}(J)/M=0\). Hence
\[
J=I,
\qquad
M=\inop_{\lex}(J)=\inop_{\lex}(I).
\]
Since the leading monomials of the elements of
\(\mathcal B\) generate \(M\), the set \(\mathcal B\) is a
Gr\"obner basis of \(I\).
\end{proof}

We now specialize the construction to the cases \(g=1\) and \(g>1\),
including the possibility \(\mathcal D=\varnothing\).

\begin{theorem}[The coprime case]\label{thm:pretzel-g1}
Assume \(g=1\).  Then
\[
\mathcal D=
\{\delta>1:\,\delta\mid q,\ \delta\mid mk,\
  \delta\nmid m,\ \delta\nmid k\}.
\]
If \(\mathcal D=\varnothing\), then \(\AG_2(L)=\{1\}\).  Otherwise,
\[
I=\ideal{\Pi_0(x),y-h(x)},
\qquad
\AG_2(L)=\{y-h(x),\Pi_0(x)\}.
\]
\end{theorem}

\begin{proof}
Here \(\lcm(m,k)=mk\).  By
\eqref{eq:pretzel-local-degree}, the positive local degrees are exactly
\(d_\delta=1\) for the displayed divisors.  If \(\mathcal D\ne\varnothing\),
then \(\mathcal S=\{0,1\}\), so
Theorem~\ref{thm:pretzel-uniform} gives the asserted ideal.  The two
polynomials are monic with leading monomials \(y\) and
\(x^{\deg\Pi_0}\), and every term of \(h\) has \(x\)-degree less than
\(\deg\Pi_0\); hence the displayed basis is reduced.  The empty case
follows from
Lemma~\ref{lem:pretzel-local-uniform}.
\end{proof}

\begin{theorem}[The case \(g>1\)]\label{thm:pretzel-gt1}
Assume \(g>1\).  Then
\[
\mathcal D=
\{\delta>1:\,\delta\mid q,\ \delta\mid\lcm(m,k)\}.
\]
If \(\mathcal D=\varnothing\), then \(\AG_2(L)=\{1\}\).  Otherwise, the
Gr\"obner basis in Theorem~\ref{thm:pretzel-uniform} consists of
\(\{\Pi_0(x),y^g-h(x)\}\), together with \(\Pi_{g-1}Q_{g-1}\) when
\(g-1\in\mathcal S\), and with \(\Pi_{g-2}Q_{g-2}\) when
\(g\ge3\) and \(g-2\in\mathcal S\).  In particular, it has at most four
elements.  Its initial ideal is generated by \(y^g\) and
\(x^{\deg\Pi_0}\), together with \(x^{\deg\Pi_{g-1}}y^{g-1}\) when
\(g-1\in\mathcal S\), and with
\(x^{\deg\Pi_{g-2}}y^{g-2}\) when
\(g\ge3\) and \(g-2\in\mathcal S\).  Moreover,
\[
\Pi_{g-1}=
\prod_{\delta:\,\delta\in\mathcal D,\ \delta\nmid m,\ \delta\nmid k}\Phi_{2\delta}.
\]
When \(g\ge3\), one also has
\[
\Pi_{g-2}=
\prod_{\delta:\,\delta\in\mathcal D,\ \delta\nmid m\ \text{or}\ \delta\nmid k}\Phi_{2\delta}.
\]
Consequently, \(\AG_2(L)\) is obtained by interreducing an
at-most-four-element basis.
\end{theorem}

\begin{proof}
If \(\delta\mid\lcm(m,k)\), then
\eqref{eq:pretzel-local-degree} gives
\(d_\delta=g-\mathbf 1_{\delta\mid m}-\mathbf 1_{\delta\mid k}\).
This is positive for \(g\ge3\).  If \(g=2\) and \(d_\delta=0\), then
\(\delta\mid m\) and \(\delta\mid k\), so \(\delta\mid g=2\), which is
impossible because \(\delta>1\) divides the odd integer \(q\).  This proves
the description of \(\mathcal D\).

For \(\delta\in\mathcal D\), the possible local degrees are
\(g,g-1,g-2\), according as
\(\delta\) divides neither, exactly one, or both of \(m\) and \(k\).
Thus \(\mathcal S\setminus\{0,g\}\subset\{g-2,g-1\}\), with the
index \(g-2\) relevant only when \(g\ge3\).  The formula for
\(\Pi_{g-1}\), and also that for
\(\Pi_{g-2}\) when \(g\ge3\), follow directly from the definition of
\(\Pi_n\).  The remaining
assertions follow from Theorem~\ref{thm:pretzel-uniform}.
\end{proof}

\begin{remark}
The calculations above concern ordered links.  For the corresponding
oriented unordered link, let \(\tau\colon R\to R\) be the involution
\(\tau(f)(x,y)=f(y,x)\).  Its invariant is the indexed family
\[
\bigl(
\mathcal G_{\lex(y>x)}(I),
\mathcal G_{\lex(y>x)}(\tau(I))
\bigr).
\]
 Simply interchanging the variables in a
displayed reduced basis generates \(\tau(I)\), but the resulting basis need
not itself be reduced for the fixed order \(\lex(y>x)\).
\end{remark}

\section{Reciprocity phenomena and a possible codimension-\(\ge2\) defect}
\label{sec:reciprocity}

The computations above exhibit frequent reciprocal behavior.
Our general reciprocity result over \(\Q\) concerns the
codimension-one part forced by Blanchfield duality.
We also give a finite Gr\"obner-basis criterion for full
ideal reciprocity.

Nonreciprocal elementary ideals with integer coefficients
were constructed by Turaev \cite{Turaev};
see also Traldi \cite[\S5]{Traldi}.
The specific second elementary ideal described by Traldi
becomes the unit ideal after extending scalars to \(\Q\),
so that example does not establish nonreciprocity over \(\Q\).

Throughout this section, we take \(\K=\Q\), and 
codimension means Krull codimension in
\(\operatorname{Spec}\Lambda\).
Let \(X\) be a connected
compact oriented \(3\)-manifold whose nonempty boundary is a disjoint union
of tori.  Put
\[
 H:=H_1(X;\Z)/\operatorname{Tor}H_1(X;\Z),
\]
choose an ordered basis \(H\cong\Z^r\), and let
\(\phi\colon\pi_1(X)\twoheadrightarrow H\) be the maximal free-abelian
quotient.  We write
\[
 \Lambda:=\Q[H]\cong\Q[t_1^{\pm1},\dots,t_r^{\pm1}],
 \qquad
 P:=\Q[t_1,\dots,t_r].
\]
The assignments \(\overline{t_i}=t_i^{-1}\) define an involution of
\(\Lambda\).  For an ideal \(I\subset\Lambda\), write
\(\overline I:=\{\overline f\mid f\in I\}\).

\subsection{Determinant-divisor reciprocity}
\label{subsec:determinant-reciprocity}

\begin{definition}\label{def:determinant-divisor}
Let \(R\) be a UFD and let \(N\) be a finitely presented \(R\)-module.
The \emph{\(k\)th determinant divisor} \(\mathfrak D_k(N)\) is the smallest
principal ideal containing \(\Fitt_k(N)\).  Equivalently, it is generated
by a greatest common divisor of any finite generating set of
\(\Fitt_k(N)\); in particular, \(\mathfrak D_k(N)=(0)\) when
\(\Fitt_k(N)=(0)\).
\end{definition}

For a \(\Lambda\)-module \(N\), the equality
\(\overline{\mathfrak D_k(N)}=\mathfrak D_k(N)\) will be called
\emph{divisorial reciprocity}; the stronger equality
\(\overline{\Fitt_k(N)}=\Fitt_k(N)\) will be called
\emph{full ideal reciprocity}.  Blanchfield's determinant divisors are the
principal gcd ideals in Definition~\ref{def:determinant-divisor}, rather
than the determinantal ideals themselves.  Blanchfield works with the
greatest common divisor of the order-\((n-i)\) minors of a presentation with
\(n\) generators.  In the notation of the finite presentation
\(R^{q}\xrightarrow{\,B\,}R^{p}\to M\to0\) introduced above, that is the gcd
of the \((p-i)\times(p-i)\) minors of \(B\), which by
Definition~\ref{def:determinant-divisor} is our \(\mathfrak D_i(M)\).  Thus
Blanchfield's \(A_i(M)\) is the object we denote \(\mathfrak D_i(M)\).
Write \(\mathcal A(X,\phi):=H_1(X,*;\Lambda)\).By Theorem\ref{thm:CF-invariance},this module agrees with  $A(\pi_1(X),\phi).$

\begin{theorem}[Blanchfield reciprocity for determinant divisors]
\label{thm:determinant-reciprocity}
For every \(k\ge0\),
\[
 \overline{\mathfrak D_k\bigl(\mathcal A(X,\phi)\bigr)}
 =\mathfrak D_k\bigl(\mathcal A(X,\phi)\bigr).
\]
In particular, the conclusion applies to the canonical meridional
coefficient system of every oriented ordered classical link.
\end{theorem}

\begin{proof}
Let \(R_{\Z}:=\Z[H]\), let \(\widetilde X\) be the maximal free-abelian
cover, and regard
\(M_{\Z}:=H_1(\widetilde X;\Z)\) as an \(R_{\Z}\)-module.  The conclusion
of Blanchfield's boundary case gives
\[
 \overline{\mathfrak D_j(M_{\Z})}=\mathfrak D_j(M_{\Z})
 \qquad(j\ge0);
\]
see \cite[Thms.~4.5 and~4.7, Cor.~4.8, Lem.~4.10,
Thm.~5.5, and Cor.~5.6]{Bla}.  Localizing \(R_{\Z}\) at the nonzero
integers gives \(\Lambda\), and flat base change yields
\[
 M:=H_1(X;\Lambda)\cong\Lambda\otimes_{R_{\Z}}M_{\Z}.
\]
Fitting ideals commute with this base change, and their principal gcd
ideals localize accordingly.  More precisely, for a multiplicative set \(S\),
\(
\Fitt_i(S^{-1}N)=S^{-1}\Fitt_i(N),
\qquad
\mathfrak D_i(S^{-1}N)=S^{-1}\mathfrak D_i(N).
\)

Since the involution fixes the localized
integers,
\[
 \overline{\mathfrak D_j(M)}=\mathfrak D_j(M)
 \qquad(j\ge0).
\]

The long exact sequence of \((X,*)\) contains
\[
 0\longrightarrow M\longrightarrow\mathcal A(X,\phi)
 \longrightarrow\mathfrak a\longrightarrow0,
 \qquad
 \mathfrak a:=\ker(\varepsilon\colon\Lambda\longrightarrow\Q),
\]
where \(\varepsilon(t_i)=1\).  Since
\(\chi(X)=\frac12\chi(\partial X)=0\), while \(b_0(X)=1\) and
\(b_3(X)=0\), one has \(r=b_1(X)=1+b_2(X)\ge1\).  Hence
\(\mathfrak a\) is a nonzero ideal of the domain \(\Lambda\), and is
therefore torsion-free of rank one.  It follows that
\[
 T\bigl(\mathcal A(X,\phi)\bigr)=T(M),
 \qquad
 \operatorname{rank}_{\Lambda}\mathcal A(X,\phi)
 =\operatorname{rank}_{\Lambda}M+1,
\]
where \(T(-)\) denotes the torsion submodule.

For a finitely generated module \(N\) of rank \(q\), Blanchfield's rank
formula \cite[Lem.~4.10]{Bla} gives
\[
 \mathfrak D_i(N)=
 \begin{cases}
  (0),&i<q,\\
  \mathfrak D_{i-q}\bigl(T(N)\bigr),&i\ge q.
 \end{cases}
\]
Put \(s:=\operatorname{rank}_{\Lambda}M\).  The preceding identities show
that \(\mathfrak D_0(\mathcal A(X,\phi))=(0)\) and, for every \(k\ge1\),
\[
 \mathfrak D_k\bigl(\mathcal A(X,\phi)\bigr)=\mathfrak D_{k-1}(M).
\]
Indeed, both sides vanish for \(k\le s\), and for \(k\ge s+1\) both are
\(\mathfrak D_{k-s-1}(T(M))\).  The theorem now follows from the
reciprocity of the determinant divisors of \(M\).
\end{proof}

\begin{corollary}[Reciprocity of the common divisor of \(\AG_k\)]
\label{cor:AG-common-divisor-reciprocity}
Let \(L\subset S^3\) be an oriented ordered link.  If
\(\AG_k(L)=\emptyset\), put \(\Delta_k=0\); otherwise, let \(\Delta_k\)
be a greatest common divisor of the elements of \(\AG_k(L)\).  Then
\[
 (\overline{\Delta_k})=(\Delta_k)
\]
as principal ideals of \(\Lambda\).  If \(\Delta_k\ne0\), there are
\(c\in\Q^\times\) and \(a_1,\dots,a_\ell\in\Z\) such that
\[
 \Delta_k(t_1^{-1},\dots,t_\ell^{-1})
 =c\,t_1^{a_1}\cdots t_\ell^{a_\ell}
  \Delta_k(t_1,\dots,t_\ell).
\]
\end{corollary}

\begin{proof}
The set \(\AG_k(L)\) generates \(\PP(E_k(L))\), and extension to
\(\Lambda\) recovers \(E_k(L)\) by
Lemma~\ref{lem:contraction-injective}.  Thus
\((\Delta_k)\Lambda=\mathfrak D_k(\mathcal A(X_L,\phi))\); since \(E_k(L)\cap P\) is saturated with respect to the variables, its nonzero gcd has no variable factor.The first
assertion follows from Theorem~\ref{thm:determinant-reciprocity}, and the
second from the description of the units of \(\Lambda\).
\end{proof}

\subsection{Height-one reciprocity and a finite criterion}
\label{subsec:height-one-reciprocity}

For \(0\ne g\in P\), define its \emph{polynomial reciprocal} by
\[
 g^\vee:=t_1^{\deg_{t_1}g}\cdots t_r^{\deg_{t_r}g}
 g(t_1^{-1},\dots,t_r^{-1})\in P.
\]

\begin{proposition}[Height-one reciprocity and reciprocal multipliers]
\label{prop:codimension-two-defect}
Fix \(k\ge0\), put
\(E:=\Fitt_k(\mathcal A(X,\phi))\), \(J:=E\cap P\), and let
\(\mathcal G:=\mathcal G_P(J)\).  Then
\[
 E_{\mathfrak p}=(\overline{E})_{\mathfrak p}
\]
for every height-one prime \(\mathfrak p\subset\Lambda\).  Moreover, there
are nonzero polynomials \(a_1,\dots,a_s\in P\) such that
\[
 \gcdop(a_1,\dots,a_s)=1,
 \qquad
 a_j g^\vee\in J
 \quad(g\in\mathcal G,\ 1\le j\le s).
\]
\end{proposition}

\begin{proof}
If \(E=(0)\), both assertions are immediate, with \(a_1=1\).  Assume
\(E\ne(0)\), and write
\(\mathfrak D_k(\mathcal A(X,\phi))=(d)\).  For a height-one prime
\(\mathfrak p\), the local ring \(\Lambda_{\mathfrak p}\) is a DVR, and
the localization of a nonzero ideal is generated by the localization of
its gcd.  The gcd of \(\overline E\) is \(\overline d\), up to a unit, and
Theorem~\ref{thm:determinant-reciprocity} gives
\((\overline d)=(d)\).  Hence
\(E_{\mathfrak p}=(\overline{E})_{\mathfrak p}\).

Set
\[
 Q:=\frac{E+\overline E}{E\cap\overline E},
 \qquad
 \mathfrak c:=\operatorname{Ann}_{\Lambda}(Q).
\]
After localization at a height-one prime, the numerator and denominator
of \(Q\) both become \(E_{\mathfrak p}\); at the zero prime, the two
nonzero ideals both become the quotient field of \(\Lambda\).  Thus
\(Q_{\mathfrak p}=0\) for every prime of height at most one, so the support
of \(Q\) has codimension at least two.

The module \(Q\) is finitely generated.  If \(q_1,\dots,q_m\) generate
it and \(Q_{\mathfrak p}=0\), choose
\(s_i\notin\mathfrak p\) with \(s_i q_i=0\).  Then
\(s_1\cdots s_m\in\mathfrak c\setminus\mathfrak p\).  Applying this at
the zero prime and at every height-one prime shows that
\(\mathfrak c\ne(0)\) and that no height-one prime contains
\(\mathfrak c\).

Since \(\Lambda\) is Noetherian, choose Laurent polynomial generators
\(c_1,\dots,c_s\) of \(\mathfrak c\).  Multiplying each \(c_j\) by a
Laurent monomial unit, we obtain generators \(b_1,\dots,b_s\in P\).  If the
gcd of the \(b_j\) in \(P\) had an irreducible factor other than a
variable, then \(\mathfrak c\) would be contained in the corresponding
height-one prime of \(\Lambda\).  Hence this gcd is a monomial in \(t_1,\dots,t_r\), up to a nonzero
scalar.  Because such a monomial is a unit in \(\Lambda\), dividing all
\(b_j\) by their common monomial factor produces nonzero polynomials
\(a_1,\dots,a_s\in\mathfrak c\cap P\) with
\(\gcdop(a_1,\dots,a_s)=1\).

Finally, \(\mathfrak cQ=0\) implies
\(\mathfrak c(E+\overline E)\subset E\cap\overline E\), and therefore
\(\mathfrak c\,\overline E\subset E\).  For \(g\in\mathcal G\), the
polynomial \(g^\vee\) is a Laurent monomial multiple of \(\overline g\).
Thus \(a_j g^\vee\in E\cap P=J\) for all \(g\in\mathcal G\) and all
\(j\).
\end{proof}

The quotient \(Q\) in the proof records the possible failure of full ideal
reciprocity, and its support has codimension at least two.  The multiplier
conclusion is weaker than full ideal reciprocity: when \(r>1\), the
condition \(\gcdop(a_1,\dots,a_s)=1\) need not imply
\(\ideal{a_1,\dots,a_s}=P\).

\begin{proposition}[A Gr\"obner-basis criterion for full ideal reciprocity]
\label{prop:groebner-full-reciprocity}
Let \(E\subset\Lambda\) be any ideal, put \(J:=E\cap P\), and let
\(\mathcal G:=\mathcal G_P(J)\).  The following are equivalent.
\begin{enumerate}[(i)]
\item \(\overline E=E\).
\item \(g^\vee\in J\) for every \(g\in\mathcal G\).
\item \(\NF_{\mathcal G}(g^\vee)=0\) for every \(g\in\mathcal G\), where
\(\NF_{\mathcal G}\) denotes the normal form with respect to
\(\mathcal G\).
\end{enumerate}
\end{proposition}

\begin{proof}
If \(\overline E=E\) and \(g\in\mathcal G\), then \(g^\vee\) is a Laurent
monomial multiple of \(\overline g\), and hence belongs to \(E\cap P=J\).
Conversely, suppose that \(g^\vee\in J\) for every \(g\in\mathcal G\).
Then \(\overline g\in E\) for every such \(g\).  By
Lemma~\ref{lem:contraction-injective},
\(E=J\Lambda=\ideal{\mathcal G}_{\Lambda}\), so
\(\overline E\subset E\).  Applying the involution gives the reverse
inclusion.  The equivalence of the last two conditions is the usual
Gr\"obner-basis membership criterion.
\end{proof}

\begin{remark}[Scope of the result]\label{rem:scope-of-reciprocity}
When \(r=1\), the Laurent polynomial ring is a PID, so divisorial
reciprocity already implies full ideal reciprocity and recovers the usual
one-variable reciprocity law.  For \(r>1\),
Proposition~\ref{prop:codimension-two-defect} controls the height-one part,
and its proof places the remaining defect in codimension at least two;
Proposition~\ref{prop:groebner-full-reciprocity} gives a finite test for
its vanishing.  The Blanchfield-based statements in this section concern
the maximal free-abelian quotient; no assertion is made for an arbitrary
epimorphism onto a free abelian group.
\end{remark}

\subsection{Full reciprocity in the computed ideals}
\label{subsec:reciprocity-examples}

For every reduced Gr\"obner basis explicitly displayed in
Subsection~\ref{subsec:low-crossing-examples}, all elements satisfy
\(g^\vee=\pm g\), with the sole exception of
\(f=t_1+t_2-1\) in \(\AG_2(\mathrm{L9a33})\).
In this case, put \(h=t_2^2-t_2+1\). Direct calculation gives
\[
 f^\vee=t_1+t_2-t_1t_2=(1-t_2)f+h,
 \qquad
 h^\vee=h.
\]
Thus \(f^\vee\) also belongs to the ideal generated by the displayed
basis. In particular, the bases explicitly displayed in
Example~\ref{ex:3pretzel-experiment} have no exceptional elements:
each of their elements satisfies \(g^\vee=-g\).
Proposition~\ref{prop:groebner-full-reciprocity} therefore shows that
all elementary ideals represented by these explicitly displayed
bases are fully reciprocal.

For the torus links of Subsection~\ref{subsec:torus-links}, one has
\[
\begin{aligned}
 D(U^{-1})&=-U^{-mn}D(U),
 &A(U^{-1})&=U^{-m(n-1)}A(U),\\
 B(U^{-1})&=U^{-n(m-1)}B(U),
 &P_\ell(U^{-1})&=U^{-(m-1)(n-1)}P_\ell(U).
\end{aligned}
\]
Since \(P_k(U)=D(U)^{\ell-1-k}\Theta(U)\) for \(k<\ell\), and
\(\overline{t_i-1}=-t_i^{-1}(t_i-1)\), the elementary ideals in
Proposition~\ref{prop:toruslink-Ek} are fully reciprocal.

For the two-component pretzel family,
Proposition~\ref{prop:pretzel-reduction} and
Lemma~\ref{lem:contraction-injective} give
\[
 E_2(L)=\ideal{A_q(x),G_k(x,y),H_m(x,y)}_{\Lambda}.
\]
The generators satisfy
\[
\begin{aligned}
 A_q(x^{-1})&=x^{1-q}A_q(x),
 &G_k(x^{-1},y^{-1})&=(xy)^{1-k}G_k(x,y),\\
 H_m(x^{-1},y^{-1})&=x^{1-m}y^{1-m}H_m(x,y).
\end{aligned}
\]
Thus \(E_2(L)\) is fully reciprocal for this family as well.

\section*{Conclusion}

For fixed coefficient and ordering data, the reduced
Gr\"obner bases studied here give canonical polynomial
representatives of higher Alexander ideals.  The explicit
computations show that these higher ideals contain
information not determined by the multivariable Alexander
polynomial.  A natural remaining question is to determine
when determinant-divisor reciprocity lifts to reciprocity
of the full ideals.

\noindent
Department of Mathematics, Institute of Science Tokyo \\
2-12-1 Ookayama, Meguro-ku, Tokyo 152-8551, Japan

\end{document}